\documentclass[ECP]{ejpecp} % replace ECP by EJP if needed.  
\usepackage{mathrsfs}

\newcommand\Seq[1]{\langle #1 \rangle}

\newcommand\I[1]{\mathbb{1}_{\{ #1 \}}}
\newcommand\II[1]{\mathbb{1}_{ #1 }}
\newcommand\st{\,;\;}

\DeclareMathOperator{\bP}{\mathbf{P}\mathopen{}}

\newcommand\dfn[1]{\textit{\textbf{#1}}}

\newcommand\ev[1]{\mathcal {#1}}
\newcommand\sX{{\mathfrak X}}
\newcommand\restrict{\!\upharpoonright\!}
\newcommand\cN{{\mathcal N}}

\newcommand\C{{\mathbb C}}
\newcommand\N{{\mathbb N}}

\newcommand\F{{\mathscr F}}
\newcommand\G{{\mathscr G}}
\newcommand\HH{{\mathscr H}}
\newcommand\CC{{\mathscr C}}
\newcommand\parts{{\mathscr A}}
\newcommand\cbuldot{{\raise.25ex\hbox{$\scriptscriptstyle\bullet$}}}
\SHORTTITLE{Tail Triviality for Determinantal Point Processes}

\TITLE{A Note on Tail Triviality for Determinantal Point Processes}

\AUTHORS{%
  Russell Lyons\footnote{Department of Mathematics, 831 E. 3rd St.,
  Indiana University,
  Bloomington, IN 47405-7106. \EMAIL{rdlyons@indiana.edu}}}


\AMSSUBJ{60K99} % Edit. Separate items with ;
\AMSSUBJSECONDARY{60G55}

\SUBMITTED{July 8, 2018}
\ACCEPTED{}

\VOLUME{0}
\YEAR{2012}
\PAPERNUM{0}
\DOI{vVOL-PID}

\ABSTRACT{We give a very short proof that determinantal point processes
have a trivial tail $\sigma$-field. This conjecture of the author
has been proved by Osada and Osada as well as by Bufetov, Qiu, and Shamov.
The former set of authors relied on the earlier result of the present author
that the conjecture held in the discrete case, as does the present short
proof.}

\begin{document}

We give a very short proof that determinantal point processes
have a trivial tail $\sigma$-field. This conjecture of Lyons \cite{Lyons:ICM}
has been proved by Osada and Osada \cite{Osada}\footnote{In fact, there is
a gap in \cite{Osada}: Lemmas 4 and 5 there do not follow from the reasoning
given and can be false when the tail $\sigma$-field is nontrivial. This gap
can be filled via reasoning similar to that used here. More precisely, one
can use the partition $\CC$ here in place of their sequence of partitions
$\mathit{\Delta}(\ell)$.}
% E.g., consider $X = (X_1, X_2, \dots)$ where $X_1$ is Bernoulli(1/2) and
% all $X_n$ are equal. Then the tail $\sigma$-field is also equal to the
% $\sigma$-field generated by $X_1$. Thus, we may consider the tail
% function $X_1$, which violates (5.2) for this measure. Although (5.2) can
% be fixed, Lemma 4 there is false: let $X$ be correlated with
% $s(A_{1, 1})$; the latter random variable not be in the tail.
as well as by Bufetov, Qiu, and Shamov \cite{BQS}.
Osada and Osada relied on the earlier result of Lyons \cite{L:det}
that the conjecture held in the discrete case, as does the present short
proof. In the discrete case and under the restrictive
assumption that the spectrum of
$K$ is contained in the open interval $(0, 1)$, Shirai and Takahashi
\cite{ShiTak:I} also proved
that the tail $\sigma$-field is trivial.
In the continuous setting, tail triviality is important in proving
pathwise uniqueness of solutions of certain infinite-dimensional
stochastic differential equations related to determinantal point processes
\cite{OsadaTan:ISDE-tail}.

Our proof here relies on an extension of Goldman's transference principle,
as elucidated in \cite{Lyons:ICM}.

\section{Goldman's Transference Principle} \label{s.goldman}

We review some definitions. See \cite{Lyons:ICM} for more details.

Let $E$ be a locally compact Polish space
(equivalently, a locally compact second countable Hausdorff space).
Let $\mu$ be a Radon measure on $E$, i.e., a Borel measure that is finite
on compact sets.
Let $\cN(E)$ be the set of Radon measures on $E$ with values in $\N \cup
\{\infty\}$.
We give $\cN(E)$ the vague topology generated by the maps $\xi \mapsto
\int f \,d\xi$ for continuous $f$ with compact support; then $\cN(E)$ is
Polish.
The corresponding Borel $\sigma$-field of
$\cN(E)$ is generated by the maps $\xi \mapsto
\xi(A)$ for Borel $A \subseteq E$.

Let $\sX$ be a simple point process on $E$, i.e., a random variable with
values in $\cN(E)$ such that $\sX(\{x\}) \in \{0, 1\}$ for all $x \in E$.
We call 
$\sX$ \dfn{determinantal} if for some measurable $K
\colon E^2 \to \C$ and all $k \ge 1$, the
function $(x_1, \dots, x_k) \mapsto \det[K(x_i, x_j)]_{i, j \le k}$
is a $k$-point intensity function of $\sX$.
In this case, we denote the law of $\sX$ by $\bP^K$.

We consider only $K$ that are locally square integrable (i.e., $|K|^2
\mu^2$ is Radon), are Hermitian (i.e., $K(y, x) = \overline{K(x, y)}$ for
all $x, y \in E$), and are positive semidefinite.
In this case, $K$ defines a positive semidefinite integral operator 
$
(Kf)(x) := \int K(x, y) f(y) \,d\mu(y)
$
on functions $f \in L^2(\mu)$ with compact support.
We consider $K$ as defined only up to changes on a $\mu^2$-null set.
For every Borel $A \subseteq E$, we denote by $\mu_A$ the measure $\mu$
restricted to Borel subsets of $A$ and by
$K_A$ the compression of $K$ to $A$, i.e., $K_A f := (K f)
\restrict A$ for $f \in L^2(A, \mu_A)$.
If $\sX$ has law $\bP^K$, then the restriction of $\sX$ to $A$ has law
$\bP^{K_A}$.
The operator $K$ is locally trace-class, i.e.,
for every compact $A \subseteq E$,
the compression $K_A$ 
is trace class, having a
spectral decomposition 
$
K_A = \sum_k \lambda^A_k\, \phi^A_k \otimes \overline{\phi^A_k}
$,
where $\Seq{\phi^A_k \st k \ge 1}$ are orthonormal eigenfunctions of $K_A$
with positive summable eigenvalues $\Seq{\lambda^A_k \st k \ge 1}$.

The following extends Goldman's transfer principle from trace-class
operators, as given in
\cite[Section 3.6]{Lyons:ICM}, to locally trace-class operators:

\begin{theorem} \label{t.transfer}
Let $\mu$ be a Radon measure on a locally compact Polish space, $E$.
Let $K$ be a locally trace-class positive contraction on $L^2(E, \mu)$.
Let $\Seq{A_i \st i \ge 1}$ be a partition of $E$ into
precompact Borel subsets of $E$. 
Then there exists a denumerable set $F$ with a partition
$\Seq{B_i \st i \ge 1}$ and a positive contraction $Q$ on $\ell^2(F)$ such that
the joint distribution of the random variables
$\Seq{\sX(A_i) \st i \ge 1}$ for $\sX \sim \bP^K$ equals the joint
$\bP^Q$-distribution of the random variables $\Seq{\sX(B_i) \st i \ge 1}$
for $\sX \sim \bP^Q$.
Moreover, we can choose $Q$ to be unitarily equivalent to $K$.
\end{theorem}

\begin{proof}
For each $i$, fix an orthonormal
basis $\Seq{w_{i, j} \st j < n_i}$ for the subspace of $L^2(E, \mu)$ of
functions that vanish outside $A_i$.
Here, $n_i \in \N \cup \{\infty\}$.
Define $B_i := \{(i, j) \st j < n_i\}$ and $F := \bigcup_i B_i$.
Let $T$ be the isometric isomorphism (i.e., unitary map) from $L^2(E, \mu)$
to $\ell^2(F)$ that sends $w_{i, j}$ to $\I{(i, j)}$.
Define $Q := T K T^{-1}$, so $Q$ is unitarily equivalent to $K$.
Note that for all $\phi \in L^2(E)$ and all $i \ge 1$, we have $T \II{A_i}
\phi = \II{B_i} T\phi$.

For $m \ge 1$, write $E_n := \bigcup_{i=1}^m A_i$ and $F_m :=
\bigcup_{i=1}^m B_i$.\vadjust{\kern2pt}
Then $K_{E_m}$ and $Q_{F_m}$ are unitarily equivalent trace-class 
operators. \vadjust{\kern2pt}
If $\Seq{\phi_{k, m} \st k \ge 1}$ are orthonormal eigenvectors of
$K_{E_m}$, so that $K_{E_m} = \sum_k \lambda^{E_m}_k \phi_{k, m} \otimes
\overline{\phi_{k, m}}$, then\vadjust{\kern2pt}
$Q_{F_m} = \sum_k \lambda^{E_m}_k T\phi_{k, m} \otimes
\overline{T\phi_{k, m}}$. Furthermore, for all $\phi, \psi \in L^2(E,
\mu)$ \vadjust{\kern2pt}
and all $i \ge 1$, we have $(\II{A_i} \phi, \psi)_{L^2(E, \mu)} =
(T\II{A_i} \phi, T\psi)_{\ell^2(F)} = 
(\II{B_i} T\phi, T\psi)_{\ell^2(F)}$.
Thus,
\cite[Theorem 3.4]{Lyons:ICM} shows that\vadjust{\kern2pt}
the $\bP^{K_{E_m}}$-distribution of $\Seq{\sX(A_i) \st i \le m}$ equals the
$\bP^{Q_{F_m}}$-distribution of $\Seq{\sX(B_i) \st i \le
m}$.\vadjust{\kern2pt}
But these are precisely 
the $\bP^K$-distribution of $\Seq{\sX(A_i) \st i \le m}$ and
the $\bP^Q$-distribution of $\Seq{\sX(B_i) \st i \le m}$, respectively.
Because these are equal for all $m \ge 1$, the desired result follows.
\end{proof}

\section{Tail Triviality: Deduction from the Discrete Case}

For a Borel set $A
\subseteq E$, let $\F(A)$ denote the $\sigma$-field 
on $\cN(E)$ generated by the functions $\xi \mapsto \xi(B)$ for Borel $B
\subseteq A$.
The \dfn{tail $\sigma$-field} is the intersection of 
$\F(E \setminus A)$ over all compact $A \subseteq E$; it is said to be
trivial when each of its events has probability 0 or 1.
For a collection $\parts$ of Borel subsets of $E$, write $\G(\parts)$ for
the $\sigma$-field generated by the functions $\xi \mapsto \xi(B)$ for $B
\in \parts$.

\begin{theorem}[conjectured by \cite{Lyons:ICM}, proved by \cite{Osada,BQS}] \label{t.ctail}
If $K$ is a locally trace-class positive contraction, then
$\bP^K$ has a trivial tail $\sigma$-field.
\end{theorem}

\begin{proof}
Consider a sequence of increasingly finer partitions $\parts_m = \{A_{m, i}
\st i \ge 1\}$ of $E$ by
precompact Borel sets $A_{m, i}$ such that the sequence $\Seq{\parts_m \st
m \ge 1}$ separates points of $E$.
(This can be obtained, for example, by writing $E$ as a countable union of
compact sets \cite[Theorem 5.3]{Kechris} and partitioning each compact set
by the fact that it is a continuous image of
the Cantor set \cite[Theorem 4.18]{Kechris}.) Then
the corresponding count $\sigma$-fields $\G(\parts_m)$ 
increase to the Borel $\sigma$-field $\F(E)$,
so L\'evy's 0-1 law tells us that for every event $\ev A \in \F(E)$, we
have $\bP\bigl(\ev A \mid \G(\parts_m)\bigr)$ converges in $L^1$
to $\II {\ev A}$. 
Similarly, if $D^{(n)} := \bigcup_{i=1}^n A_{1, i}$ and $\G_m^{(n)} :=
\G\bigl(\{A_{m, i} \st A_{m, i} \cap
D^{(n)} = \varnothing,\, i \ge 1\}\bigr)$, then for each $n$ and all
$\ev A \in \F(E \setminus D^{(n)})$, we have
$\bP(\ev A \mid \G_m^{(n)})$ converges in $L^1$ to $\II {\ev A}$
as $m \to\infty$. 
In particular, if $\ev A$ is a tail event, then there is a sequence $m_n
\to\infty$ such that $\bP(\ev A \mid \G_{m_n}^{(n)})$ converges 
in $L^1$ to $\II {\ev A}$ as $n \to\infty$. 
It follows that 
$\ev A$ belongs to the completion of the $\sigma$-field $\bigvee_{n \ge k}
\G_{m_n}^{(n)}$ for each $k \ge 1$.

Now let $\ev A$ be a tail event and $\Seq{m_n \st n \ge 1}$ be such a
sequence. Let $\CC := \Seq{C_k \st k \ge 1}$ be the parts of the partition
of $E$ generated by $\{A_{m_n, i} \st A_{m_n, i} \cap D^{(n)} =
\varnothing,\, n \ge 1,\, i \ge 1\}$.
Write $\HH_n := \G\bigl(\{C_k \st k \ge n\}\bigr)$.
Then $\ev A$ belongs to the completion of the $\sigma$-field
$\HH_n$ for each $n \ge 1$, whence $\bP(\ev A \mid \bigcap_{n \ge 1} \HH_n)
= \lim_{n \to\infty} \bP(\ev A \mid \HH_n)
= \II{\ev A}$ a.s.\ by L\'evy's downwards theorem. 
%Therefore, there is an event $\ev B$ of probability 0 such that $\ev A \xor
%\ev B \in \bigcap_{n \ge 1} \HH_n$.
% OR: Write $A = B_n \xor N_n$, where $B_n \in \HH_n$ and $B_n$ is a null
% set. Then $\limsup_n B_n \in \bigcap_n \HH_n$ and differs from $A$ by a
% null set.
By Theorem~\ref{t.transfer}, there is
a partition $\Seq{B_k \st k \ge 1}$ of a denumerable set $F$
and a positive contraction $Q$ on $\ell^2(F)$ such that the
$\bP^{Q}$-distribution of 
$\Seq{\sX(B_k) \st k \ge 1}$ equals the
$\bP^K$-distribution of $\Seq{\sX(C_k) \st k \ge 1}$.
Let $\HH'_n := \G\bigl(\{B_k \st k \ge n\}\bigr)$. Then $\bigcap_{n \ge 1}
\HH'_n$ is contained in the tail $\sigma$-field $\bigcap_{B \text{ finite}}
\F(F \setminus B)$. Since the latter is trivial by
\cite[Theorem 7.15]{L:det}, so is the former.  
Therefore, so is 
$\bigcap_{n \ge 1} \HH_n$, whence $\ev A$ has probability 0 or 1.
\end{proof}

%%%%%%%%%%%%%%%%%%%%%%%%%%%%%%%%%%%%%%%%%%%%%%%%%%%%%%%%%%%%%%%%%%%
%%                                                               %%
%% Use the two commands below for producing your bibliography    %%
%% with bibtex, then comment again the commands and include the  %%
%% content of the .bbl file in this file below the commands.     %%
%%                                                               %%
%%%%%%%%%%%%%%%%%%%%%%%%%%%%%%%%%%%%%%%%%%%%%%%%%%%%%%%%%%%%%%%%%%%

%\bibliographystyle{/usr/local/texlive/2015/texmf-dist/bibtex/bst/amscls/amsplain}
%\bibliography{\jobname}

\providecommand{\bysame}{\leavevmode\hbox to3em{\hrulefill}\thinspace}
\providecommand{\MR}{\relax\ifhmode\unskip\space\fi MR }
\providecommand{\MRhref}[2]{%
  \href{http://www.ams.org/mathscinet-getitem?mr=#1}{#2}
}
\providecommand{\href}[2]{#2}
\begin{thebibliography}{1}

\bibitem{BQS}
Alexander~I. Bufetov, Yanqi Qiu, and Alexander Shamov, \emph{Kernels of
  conditional determinantal measures and the {L}yons--{P}eres conjecture},
  (2016), Preprint, \ARXIV{1612.06751}.

\bibitem{Kechris}
Alexander~S. Kechris, \emph{Classical descriptive set theory}, Graduate Texts
  in Mathematics, vol. 156, Springer-Verlag, New York, 1995. \MR{1321597}

\bibitem{L:det}
Russell Lyons, \emph{Determinantal probability measures}, Publ. Math. Inst.
  Hautes \'Etudes Sci. \textbf{98} (2003), no.~1, 167--212. \MR{2031202
  (2005b:60024)}

\bibitem{Lyons:ICM}
\bysame, \emph{Determinantal probability: {B}asic properties and conjectures},
  Proceedings of the {I}nternational {C}ongress of {M}athematicians. {V}olume
  {IV} (Sun~Young Jang, Young~Rock Kim, Dae-Woong Lee, and Ikkwon Yie, eds.),
  Kyung Moon Sa, Seoul, 2014, Invited lectures, Held in Seoul, August 13--21,
  2014, pp.~137--161. \MR{3727606}

\bibitem{Osada}
Hirofumi Osada and Shota Osada, \emph{Discrete approximations of determinantal
  point processes on continuous spaces: Tree representations and tail
  triviality}, J. Stat. Phys. \textbf{170} (2018), no.~2, 421--435.
  \MR{3744393}

\bibitem{OsadaTan:ISDE-tail}
Hirofumi Osada and Hideki Tanemura, \emph{Infinite-dimensional stochastic
  differential equations and tail $\sigma$-fields},  (2014), Preprint,
  \ARXIV{1412.8674}.

\bibitem{ShiTak:I}
Tomoyuki Shirai and Yoichiro Takahashi, \emph{Random point fields associated
  with certain {F}redholm determinants. {I}. {F}ermion, {P}oisson and boson
  point processes}, J. Funct. Anal. \textbf{205} (2003), no.~2, 414--463.
  \MR{2018415}

\end{thebibliography}

% add below the content of your .bbl file produced by bibtex.

\providecommand{\bysame}{\leavevmode\hbox to3em{\hrulefill}\thinspace}
\providecommand{\MR}{\relax\ifhmode\unskip\space\fi MR }
\providecommand{\MRhref}[2]{%
  \href{http://www.ams.org/mathscinet-getitem?mr=#1}{#2}
}
\providecommand{\href}[2]{#2}

%\begin{thebibliography}{99}

%\end{thebibliography}

\end{document}